\documentclass[leqno,10pt,a4paper]{article}
\usepackage{amsmath,amstext, amsthm, amssymb}

\usepackage{hyperref}
\usepackage[totalheight=22 true cm, totalwidth=14 true cm]{geometry}
\usepackage[english]{babel}

\newtheorem{thmintro}{Theorem}

\newtheorem{thm}{Theorem}[section]
\newtheorem{cor}[thm]{Corollary}
\newtheorem{lemma}[thm]{Lemma}
\newtheorem{prop}[thm]{Proposition}

\theoremstyle{remark}
\theoremstyle{definition}
\newtheorem{defn}[thm]{Definition}

\newtheorem{rmk}[thm]{Remark}
\newtheorem{exa}[thm]{Example}
\newtheorem{notation}[thm]{Notation}

\numberwithin{equation}{thm}

\def\beq{\begin{equation}}
\def\eeq{\end{equation}}

\def\crash#1{}

\def\Z{{\mathbb Z}}
\def\Q{{\mathbb Q}}
\def\R{{\mathbb R}}
\def\C{{\mathbb C}}

\def\l{\left}
\def\r{\right}
\def\[[{\l[\l[}
\def\]]{\r]\r]}
\def\p{\prime}

\def\sgq{\sigma_q}

\def\Sgq{\Sigma_q}

\def\ord{{\rm ord}}

\def\cf{\emph{cf. }}
\def\ie{\emph{i.e. }}
\def\ds{\displaystyle}

\def\cB{{\mathcal B}}

\def\cM{{\mathcal M}}
\def\cN{{\mathcal N}}

\def\cL{{\mathcal L}}
\def\cO{{\mathcal O}}
\def\cQ{{\mathcal Q}}
\def\cR{{\mathcal R}}

\def\bC{{\mathbf C}}
\def\bK{{\mathbf K}}

\def\wtilde{\widetilde}
\def\what{\widehat}
\def\ul{\underline}

\def\a{\alpha}
\def\be{\beta}

\def\sg{\sigma}

\def\la{\lambda}
\def\La{\Lambda}

\def\Exp{{\mathrm E}{\mathrm x}{\mathrm p}}
\def\qdiffKa{$q$-$Diff_\bK^a$}
\def\qdiffKna{$q$-$Diff_{\bK_n}^a$}
\def\qdiffKaa{$q$-$Diff_\bK^{aa}$}

\author{Lucia Di Vizio
\thanks{Work partially supported by ANR, contract ANR-06-JCJC-0028}}

\date{{\small Institut de Math\'{e}matiques de Jussieu,
Topologie et g\'{e}om\'{e}trie alg\'{e}briques,}\\
{\small Case 7012, 2, place Jussieu, 75251 Paris Cedex 05, France.}\\
{\small e-mail: {\tt divizio@math.jussieu.fr}.}}

\title{
Local analytic
classification\\ of $q$-difference equations with $|q|=1$}

\begin{document}
\maketitle


\tableofcontents


\section*{Introduction}
\addcontentsline{toc}{section}{Introduction}

For an
algebraic complex semisimple group $G$ and for a fixed
$q\in\C^\ast=\C\smallsetminus\{0\}$, $|q|\neq 1$,
V. Baranovsky and V. Ginzburg prove the following statement:

\begin{thmintro}[{\cite[Thm. 1.2]{BaranovskyGinzburg}}]\label{thmintro:BGorig}
There exists a natural bijection between
the isomorphism classes of holomorphic principal semistable $G$-bundles on the elliptic curve
$\C/q^\Z$ and  the
integral twisted conjugacy classes of the points of
$G$ that are rational over $\C((x))$.
\end{thmintro}

The \emph{twisted conjugation} is an action of $G(\C((x)))$ on itself defined by
$$
(g(x),a(x))\longmapsto {}^{g(x)}a(x)=g(qx)a(x)g(x)^{-1}\,.
$$
An equivalence class is call \emph{integral} when it contains a point of $G$ rational over $\C[[x]]$.
\par
As the authors themselves point out, this result is better understood in terms
of $q$-difference equations.
If $G=Gl_\nu$, then the integral twisted
conjugacy classes of $G(\C((x)))$
correspond exactly to the isomorphism classes of
formal regular singular $q$-difference systems.
In fact, consider a $q$-difference equation
$$
Y(qx)=B(x)Y(x)\,,
\hskip 10 pt\hbox{with~}B(x)\in Gl_\nu(\C((x)))\,.
$$
Then this system is regular singular if there exists $G(x)\in Gl_\nu(\C((x)))$ such
that $B^\p(x)=G(qx)B(x)G(x)^{-1}\in Gl_\nu(\C[[x]])$. In this case if $Y(x)$ is a solution of
$Y(qx)=B(x)Y(x)$ in some $q$-difference algebra extending $\C((x))$, then $W(x)=G(x)Y(x)$
is solution of the system $W(qx)=B^\p(x)W(x)$.

\par
Y. Soibelman and V. Vologodsky in  \cite{SoibelmanVologodsky}
use an analogous approach, \emph{via} $q$-difference
equations, to understand vector bundles on non commutative elliptic curves.
Their classification,
and hence the classification of analytic $q$-difference systems, with $|q|=1$,
is a step in Y. Manin's \emph{Alterstraum} \cite{Manin},
for understanding real multiplication through non commutative geometry.
\par
In \cite{SoibelmanVologodsky}, the authors identify the category of coherent modules on the
elliptic curve $\C^\ast/q^\Z$, for $q\in\C^\ast$,
not a root of unity, to the category of
$\cO(\C^\ast)\rtimes q^\Z$-modules of finite presentation over the ring $\cO(\C^\ast)$
of holomorphic functions on $\C^\ast$
(\cf \cite[\S2,\S3]{SoibelmanVologodsky}), both in the classic
(\ie $|q|\neq 1$) and in the non commutative (\ie $|q|=1$) case.
For $|q|=1$, they study, under convenient diophantine assumptions,
its Picard group and make a list of
simple objects.
In the second part of the paper,
they focus on the classification of formal
analogous objects defined over $\C((x))$, namely of
$\C((x))$-finite vector spaces
$M$ equipped with a semilinear invertible operator $\Sgq$, such that
$\Sgq(f(x)m)=f(qx)\Sgq(m)$, for any $f(x)\in\C((x))$ and any $m\in M$.
\par
In this paper, we establish, under convenient diophantine assumptions, an analytic
classification of $q$-difference modules over the field $\C(\{x\})$ of germs of
meromorphic functions at zero, proving some analytic analogs of the results in
\cite{SoibelmanVologodsky} and in \cite{BaranovskyGinzburg}.

$$\ast\ast\ast$$

We fix $q\in\C$, $|q|=1$, not a root of unity. Let $\cB_q$
(resp. $\what\cB_q$) be the category of $q$-difference module over $\bK:=\C(\{x\})$ (resp.
$\what\bK:\C((x))$).
Let us consider a $q$-difference module over $\bK$ and fix a basis $\ul e$
such that $\Sgq\ul e=\ul e B(x)$, with $B(x)\in Gl_\nu(\bK)$.
If it is a regular singular, or equivalently if its Newton polygon has only
the zero slope (\cf \S\ref{subsec:NP}), then
we can choose a basis $\ul f$ of $M\otimes_\bK \C((x))$ such that
$\Sgq\ul f=\ul fB^\p$ and $B^\p$ is a constant matrix in $Gl_\nu(\C)$.
When $|q|\neq 1$ we do not need to extends the scalars to $\C((x))$ and we can
find such a basis $\ul f$ over $\bK$.
When $|q|=1$ this is not possible in general because
of some small divisors appearing in the construction of the basis change.
\par
The dichotomy between the ``$|q|\neq 1$'' and the ``$|q|=1$''case becomes even more evident
when the Newton polygons has more than one slope.
In fact, let $(M,\Sgq)$ be an object of $\cB_q$, with a Newton polygon having
slopes $\mu_1<\cdots<\mu_k$, such that the projection of $\mu_i\in\Q$ on the $x$-axis
has length $r_i\in\Z_{>0}$, and let $(\what M,\what \Sgq)$ be the formal
object in $\what\cB_q$ obtained by scalar extension to $\what\bK$.
If $|q|=1$, the analytic isomorphism classes in $\cB_q$ corresponding to the formal isomorphism
class of $(\what M,\what \Sgq)$ in $\what\cB_q$
form a complex affine variety of dimension (\cf \cite{RSZcrasVariete},
\cite{SauloyNotesVariete},
\cite{vdPutReversatToulouse})
$$
\ds\sum_{1\leq i<j\leq k}r_ir_j(\mu_j-\mu_i)\,.
$$
When $|q|=1$ it may happen that the formal and analytic isomorphism classes correspond
one-to-one or that the situation gets much more complicated than the one described
above for $|q|\neq 1$.
\par
The object of this paper is the characterization of
the largest full subcategory $\cB_q^{iso}$ of $\cB_q$ such that
the extension of scalars ``$-\otimes_\bK\C((x))$'' induces an equivalence of categories
of $\cB_q^{iso}$ onto its image in $\what\cB_q$
(\ie that the formal and analytic isomorphism classes coincide).
\par
The objects of $\cB_q^{iso}$ are $q$-difference
modules over $\bK$
satisfying a diophantine condition (\cf \S\ref{subsec:admissibleoperator} and
\S\ref{subsec:simpleobjects} below). They admit a decomposition
associated to their Newton polygon, namely they are \emph{direct sum}
of $q$-difference modules,
whose Newton polygon has one single slope.
The indecomposable objects, \ie those objects that cannot be written as direct sum of
submodules, are obtained by iterated non trivial extension of a simple objet by itself.
The simple objects are all obtained by scalar restriction to $\bK$
from rank 1 $q^{1/n}$-difference objects over $\bK(t)$, $x=t^n$,
associated to equations of the form $y(q^{1/n}t)=\frac{\la}{t^\mu}y(t)$, with
$\la\in\C^\ast$ and $\mu\in\Z$, with $(\mu,n)=1$.
\par
If we call $\cB_q^{iso,f}$ the subcategory of $\cB_q^{iso}$ of the objects
whose Newton polygon has only one slope equal to
zero\footnote{The notation $\what\cB_q^{iso,f}$ reminds that this is a
category of \emph{fuchsian} $q$-difference modules.}, then we have:

\begin{thmintro}
The category $\cB_q^{iso}$ is equivalent to the category of $\Q$-graded objects
of $\cB_q^{iso,f}$, \ie each object of $\cB_q^{iso}$ is a direct sum indexed on $\Q$
of objects of $\cB_q^{iso,f}$ and the morphisms of $q$-difference modules respect the grading.
\end{thmintro}

Notice that Soibelman and Vologodsky in \cite{SoibelmanVologodsky}
prove exactly the same statement for the category of formal $q$-difference module $\what\cB_q$.
Moreover we have:

\begin{thmintro}
The category $\cB_q^{iso,f}$ is equivalent to the category of finite dimensional
$\C^\ast/q^\Z$-graded complex vector spaces $V$ endowed with nilpotent operators
which preserves the grading, that moreover have the following property:
\begin{quote}
Let $\la_1,\dots,\la_n\in\C^\ast$ be a set of representatives of the classes
of $\C^\ast/q^\Z$ corresponding to non zero homogeneous components of $V$.
The series $\Phi_{(q;\ul\La)}(x)$ (defined in Definition \ref{defn:admissible})
is convergent.
\end{quote}
\end{thmintro}

Combined with the result proved in \cite{SoibelmanVologodsky} that the objects of
$\what\cB_q$ of slope zero form a category which is equivalent to the category of
$\C/q^\Z$-graded complex vector spaces equipped with a nilpotent operator respecting the grading,
this gives a characterization of the image of $\cB_q^{iso,f}$ in $\what\cB_q$ \emph{via} the
scalar extension.

To prove the classification described above, one only need to study the small divisor problem
(\cf \S\ref{sec:smalldiv}). Once this is done, the techniques used are similar to the techniques used in
$q$-difference equations theory for $|q|\neq 1$ (\cf the papers of
F. Marotte and Ch. Zhang \cite{MarotteZhang}, J. Sauloy \cite{sauloyfiltration},
M. van der Put and M. Reversat \cite{vdPutReversatToulouse}, that have their roots in
the work of G. D. Birkhoff and P.E. Guether \cite{BirkhoffGuenther} and C.R. Adams
\cite{Adams}). The statements we have cited in this introduction are actually consequences
of analytic factorizations properties of $q$-difference linear operators (\cf \S\ref{sec:factorisation}
below). Finally, we point out a work in progress by
C. De Concini, D. Hernandez, N. Reshetikhin applying the analytic classification
of $q$-difference modules with $|q|\neq 1$ to the study of quantum affine algebras.
The study of $q$-difference equations with $|q|=1$ should help to complete the theory.
\par
A last remark: the greatest part of the statements proved in this paper are true
also in the ultrametric case, therefore we will mainly work over an algebraically
closed normed field $\bC,|~|$.

\medskip
{\bfseries Acknowledgement.}
I would like to thank the Centre de Recerca Matem\`{a}tica
of the Universitat Auton\`{o}ma de Barcelona for the hospitality, D. Sauzin  and J.-P. Marco for
answering to all my questions on small divisor problem, and D.
Bertrand, Y. Manin and
M. Marcolli for the interest they have manifested for this work.
I also would like to thanks D. Hernandez: it is mainly because of his questions
that I started working on the present paper.


\section{A small divisor problem}
\label{sec:smalldiv}

Let:
\begin{trivlist}
\item $\bullet$
$q=\exp(2i\pi\omega)$, with
$\omega\in(0,1)\smallsetminus\Q$;

\item $\bullet$
$\la=\exp(2i\pi\a)$, with $\a\in(0,1]$ and $\la\not\in q^{\Z_{\leq 0}}$.
\end{trivlist}
We want to study the convergence of the
$q$-hypergeometric series
\beq\label{eq:mainseries}
\phi_{(q;\la)}(x)=\sum_{n\geq 0}\frac{x^n}{(\la;q)_n}\in\C[[x]]\,,
\eeq
where the \emph{$q$-Pochhammer symbols} appearing at
the denominator of the coefficients of $\phi_{(q;\la)}(x)$ are defined by:
$$
\l\{\begin{array}{ll}
(\la;q)_0=1\,,\\
(\la;q)_n=(1-\la)(1-q\la)\cdots(1-q^{n-1}\la)\,,& \hbox{for $n\geq 1$}\,.
\end{array}\r.
$$
This is a well-known problem in complex dynamics.
Nevertheless we give here some
proofs that already contains the problems and the ideas used in the sequel:

\begin{prop}\label{prop:smalldivisors}
Suppose that $\la\not\in q^\Z$.
The series $\phi_{(q;\la)}(x)$ converges if and only if
both the series $\sum_{n\geq 0}\frac{x^n}{(q;q)_n}$ and
the series $\sum_{n\geq 0}\frac{x^n}{1-q^n\la}$ converge.
Under these assumptions the radius of convergence of $\phi_{(q;\ul\la)}(x)$ is at least:
$$
R(\omega)\inf\l(1,r(\a)\r)\,,
$$
where $R(\omega)$ (resp. $r(\a)$) is the radius of convergence of
$\sum_{n\geq 0}\frac{x^n}{(q;q)_n}$
(resp. $\sum_{n\geq 0}\frac{x^n}{1-q^n\la}$).
\end{prop}

\begin{rmk}\label{rmk:q-exp}
If $\la\in q^{\Z_{>0}}$, the series $\phi_{(q;\la)}(x)$ is defined and its radius of convergence
is equal to $R(\omega)$. Estimates and lower bounds for $R(\omega)$ and $r(\a)$ are discussed
in the following subsection.
\end{rmk}

The proof of the Proposition \ref{prop:smalldivisors} obviously
follows from the lemma below, which is a $q$-analogue of a special case of the
Kummer transformation formula:
$$
\sum_{n\geq 0}\frac{x^s}{(1-\a)(2-\a)\cdots(n-\a)}
=\a\exp(x)\sum_{n\geq 0}\frac{(-x)^n}{n!}\frac1{\a-n}\,,
$$
used in some estimates for $p$-adic Liouville numbers
\cite[Ch.VI, Lemma 1.1]{DGS}.

\begin{lemma}[{\cite[Lemma 20.1]{DVdwork}}]
We have the following formal identity:
$$
\phi_{(q;q\la)}\l(x\r)
=\ds\sum_{n\geq 0}\frac{x^n}{(1-q\la)\cdots(1-q^n\la)}\\ \\
=\ds(1-\la)\l(\sum_{n\geq 0}\frac{x^n}{(q;q)_n}\r)
\l(\sum_{n\geq 0}q^{\frac{n(n+1)}{2}}\frac{(-x)^n}{(q;q)_n}\frac1{1-q^n\la}\r)\ .
$$
\end{lemma}

\begin{proof}
We set $x=(1-q)t$, $[n]_q=1+q+\dots+q^{n-1}$ and $[0]_q=1$, $[n]_q^!=[n]_q[n-1]_q^!$.
Then we have to show the identity:
$$
\phi_{(q;q\la)}\l((1-q)t\r)
=\ds(1-\la)\l(\sum_{n\geq 0}\frac{t^n}{[n]_q^!}\r)
\l(\sum_{n\geq 0}q^{\frac{n(n+1)}{2}}\frac{(-t)^n}{[n]_q^!}\frac 1{1-q^n\la}\r)\,.
$$
Consider the $q$-difference operator $\sgq:t\longmapsto qt$.
One verifies directly that the series $\Phi(t):=\phi_{(q;q\la)}\l((1-q)t\r)$
is solution of the $q$-difference operator:
$$
\cL=\big[\sgq-1\big]\circ\big[\la\sgq-\l((q-1)t+1\r)\big]
=\la\sgq^2-\l((q-1)qt+1+\la\r)\sgq+(q-1)qt+1\,,
$$
in fact:
$$
\begin{array}{rcl}
{\mathcal L}\Phi(t)
&=&\big[\sgq-1\big]\circ\big[\la\sgq-\l((q-1)t+1\r)\big]\Phi(t)\\ \\
&=&\big[\sgq-1\big](\la-1)=0\,.
\end{array}
$$
Since the roots of the characteristic equation\footnote{\ie
the equation whose coefficients are the constant terms of the coefficients of the
$q$-difference operator. For a complete description of its construction and properties
\cf \S\ref{sec:factorisation}.}
$\la T^2-(\la+1)T+1=0$ of $\cL$
are exactly $\la^{-1}\not\in q^\Z$ and $1$, any solution of ${\mathcal L}y(t)=0$ of the
form $1+\sum_{n\geq 1}a_nt^n\in\C\[[t\]]$ must coincide with $\Phi(t)$.
Therefore, to finish the proof of the lemma, it is enough to verify that
$$
\Psi(t)=(1-\la)\l(\sum_{n\geq 0}\frac{t^n}{[n]_q^!}\r)
\l(\sum_{n\geq 0}q^{\frac{n(n+1)}{2}}\frac{(-t)^n}{[n]_q^!}\frac1{1-q^n\la}\r)
$$
is a solution of ${\mathcal L}y(t)=0$ and that $\Psi(0)=1$.
\par
Let $e_q(t)=\sum_{n\geq 0}\frac{t^n}{[n]_q^!}$. Then $e_q(t)$ satisfies the $q$-difference equation
$$
e_q(qt)=\l((q-1)t+1\r)e_q(t)\,,
$$
hence
$$
\begin{array}{rcl}{\mathcal L}\circ e_q(t)
&=&\big[\sgq-1\big]\circ e_q(qt)\circ \big[\la\sgq-1\big]\\ \\
&=&e_q(t)\l((q-1)t+1\r)\big[\l((q-1)qt+1\r)\sgq-1\big]\circ \big[\la\sgq-1\big]\\ \\
&=&(\ast)\big[\l((q-1)qt+1\r)\sgq-1\big]\circ \big[\la\sgq-1\big]\,,
\end{array}
$$
where we have denoted with $(\ast)$ a coefficient in $\C(t)$, not depending on $\sgq$.
\par
Consider the series $E_q(t)=\sum_{n\geq 0}q^{\frac{n(n+1)}{2}}\frac{t^n}{[n]_q^!}$, which satisfies
$$
\l(1-(q-1)t\r)E_q(qt)=E_q(t)\,,
$$
and the series
$$
g_\la(t)=\sum_{n\geq 0}q^{\frac{n(n+1)}{2}}\frac{(-t)^n}{[n]_q^!}\frac1{1-q^n\la}\,.
$$
Then
$$
\begin{array}{rcl}{\mathcal L}\circ e_q(t)g_\la(t)
&=&(\ast)\big[\l((q-1)qt+1\r)\sgq-1\big]\circ \big[\la\sgq-1\big]g_\la(t)\\ \\
&=&(\ast)\big[\l((q-1)qt+1\r)\sgq-1\big] E_q(-qt)\\ \\
&=&(\ast)\big[\l((q-1)qt+1\r)E_q(-q^2t)-E_q(-qt)\big]\\ \\
&=&0\,.
\end{array}
$$
It is enough to observe that $e_q(0)g_\la(0)=\frac1{1-\la}$
to conclude that the series $\Psi(t)=(1-\la)e_q(t)g_\la(t)$ coincides with $\Phi(t)$.
\end{proof}

\begin{rmk}\label{rmk:casopadico}
Let $(C,|~|)$ be a field equipped with an ultrametric norm and let
$q\in C$, with $|q|=1$ and $q$ not a root of unity. Then the formal
equivalence in Lemma \ref{lemma:qn-la} is still true. The series
$\sum_{n\geq 0}\frac{x^n}{(q;q)_n}$ is convergent for any $q\in C$
such that $|q|=1$ (\cf \cite[\S2]{DVAndre}). On the other side the
series $\sum_{n\geq 0}\frac{x^n}{q^n-\la}$ is not always convergent.
If $\l|\frac{\la-1}{q-1}\r|<1$ then its radius of convergence
coincides with the radius of convergence of the series $\sum_{n\geq
0}\frac{x^n}{n-\a}$, where $\a=\frac{\log\la}{\log q}$ (\cf
\cite[\S19]{DVdwork}, \cite[Ch. VI]{DGS}), otherwise it converges
for $|x|<1$.
\end{rmk}

\subsection{Some remarks on Proposition \ref{prop:smalldivisors}}

Let us make some comments on
the convergence of the series $\sum_{n\geq 0}\frac{x^n}{(q;q)_n}$ and
$\sum_{n\geq 0}\frac{x^n}{1-q^n\la}$.
A first contribution to the study
the convergence of the series $\sum_{n\geq 0}\frac{x^n}{(q;q)_n}$ can be found in
\cite{HW}. The subject has been studied in detail in \cite{LubinskyCandian}.

\begin{defn} (\cf for instance \cite[\S4.4]{MarmiIntSmallDiv})
Let $\l\{\frac{p_n}{q_n}\r\}_{n\geq 0}$
be the convergents of $\omega$, occurring in its continued fraction expansion.
Then the \emph{Brjuno function $\cB$} of $\omega$ is defined by
$$
\cB(\omega)=\sum_{n\geq 0}\frac{\log q_{n+1}}{q_n}
$$
and $\omega$ is a \emph{Brjuno number} if $\cB(\omega)<\infty$.
\end{defn}

Now we are ready to recall the well-known theorem:

\begin{thm}[Yoccoz lower bound, {\cf \cite{Yoccoz}, \cite[Thm. 2.1]{CarlettiMarmi},
\cite[Thm. 5.1]{MarmiIntSmallDiv}}]
If
$\omega$ is a Brjuno number then
the series $\sum_{n\geq 0}\frac{x^n}{(q;q)_n}$ converges.
\par
Moreover its radius of convergence is bounded from below by
$e^{-B(\omega)-C_0}$, where $C_0>0$ is an universal constant (\ie independent of $\omega$).
\end{thm}

\begin{proof}[Sketch of the proof.]
Suppose that $\omega$ is a Brjuno number, then our statement is much easier than the ones
cited above and its actually an immediate consequence of the Davie's lemma
(\cf \cite[Lemma 5.6 (c)]{MarmiIntSmallDiv} or \cite[Lemma B.4,3)]{CarlettiMarmi}).
\end{proof}

We set $\|x\|_\Z=\inf_{k\in\Z}|x+k|$.
Then, as far as the series $\sum_{n\geq 0}\frac{x^n}{1-q^n\la}$ is concerned, we have:

\begin{lemma}\label{lemma:qn-la}
The following assertions are equivalent:
\begin{enumerate}
\item
The series $\sum_{n\geq 0}\frac{x^n}{1-q^n\la}$ is convergent.

\item
$\ds\limsup_{n\to\infty}
\frac{\log |1-\la q^n|^{-1}}n<+\infty$.

\item
$\ds\liminf_{n\to\infty}\|n\omega+\a\|_\Z^{1/n}>0$.
\end{enumerate}
\end{lemma}

\begin{proof}
The equivalence between 1. and 2. is straightforward. Let us prove the equivalence
``$1\Leftrightarrow 3$'' (using a really classical argument).
\par
Notice that for any $x\in[0,1/4]$ we have
$f(x):=sin(\pi x)-x\geq 0$, in fact
$f(0)=0$ and $f^\p(x)=\pi cos(\pi x)-1\geq 0$.
Therefore we conclude that the following inequality
holds for any $x\in[0,1/2]$:
$$
\sin(\pi x)>\min\big(x,1/4\big)\,.
$$
This implies that:
$$
|q^n\la-1|=\l|\exp\l(2i\pi (n\omega+\a)\r)-1\r|=2\sin\l(\pi\|n\omega+\a\|_\Z\r)
\in\Big[\min\l(2\|n\omega+\a\|_\Z,1/2\r),2\pi\|n\omega+\a\|_\Z\Big[\,
$$
and ends the proof.
\end{proof}

\begin{rmk}
A basic notion in complex dynamics is that a number $\a$ is diophantine with respect to
another number, say $\omega$. If $\a$ is diophantine with respect to $\omega$ then $\a$ and $\omega$
have the properties of the previous lemma.
It is known that, for a given $\omega\in[0,1]\smallsetminus \Q$,
the complex numbers $\exp(2i\pi\a)$ such that $\a$ is diophantine with respect to $\omega$
form a subset of the unit circle of full Lebesgue measure; \cf \cite[\S 1.3]{BernikDodson}.
\end{rmk}

\subsection{A corollary}

Let:
\begin{trivlist}
\item $\bullet$
$q=\exp(2i\pi\omega)$, with
$\omega\in(0,1)\smallsetminus\Q$;

\item $\bullet$
$m\in\Z_{>0}$ and $\la_i=\exp(2i\pi\a_i)$, for $i=1,\dots,m$,
with $\a_i\in(0,1]$ and $\la_i\not\in q^{\Z}$.
\end{trivlist}
For further reference we state the corollary below which is an immediate consequence
of Proposition \ref{prop:smalldivisors}:

\begin{cor}\label{cor:smalldivisors}
Let $\La=(\la_1,\dots,\la_m)$. The series
\beq\label{eq:smalldivisors}
\phi_{(q;\La)}(x)=\sum_{n\geq 0}\frac{x^n}{{(\la_1;q)_n\cdots(\la_m;q)_n}}\in\C[[x]]\,
\eeq
converges if and only if
both the series $\sum_{n\geq 0}\frac{x^n}{(q;q)_n}$ and
the series $\sum_{n\geq 0}\frac{x^n}{1-q^n\la_i}$, for $i=1,\dots,m$, converge.
Under these assumptions the radius of convergence of $\phi_{(q;\La)}(x)$ is at least:
$$
R(\omega)^m\cdot\prod_{i=1}^m\inf\l(1,r(\a_i)\r)\,.
$$
\end{cor}


\section{Analytic factorization of $q$-difference operators}
\label{sec:factorisation}

\begin{notation}
Let $(\bC,|~|)$ be either the field of complex numbers with the usual norm
or an algebraically
closed field with an ultrametric norm. We fix $q\in\bC$, such that $|q|=1$
and $q$ is not a root of unity,
and a set of elements $q^{1/n}\in\bC$ such that $(q^{1/n})^n=q$.
If $\bC=\C$ then let $\omega\in(0,1]\smallsetminus\Q$ be such that $q=\exp(2i\pi\omega)$.
\par
\emph{We suppose that the series $\sum_{n\geq 0}\frac{x^n}{(q;q)_n}$ is convergent},
which happens for instance if $\omega$ is a Brjuno number.
\end{notation}

The contents of this section
is largely inspired by \cite{sauloyfiltration},
where the author proves an analytic classification result for $q$-difference
equations with $|q|\neq 1$: the
major difference is the small divisor problem that the assumption $|q|=1$ introduces.
Of course, once the small divisor problem is solved, the techniques are the same.
For this reason some proofs will be only sketched.

\subsection{The Newton polygon}
\label{subsec:NP}

We consider a $q$-difference operator
$$
\cL=\sum_{i=0}^\nu a_i(x)\sgq^i\in\bC\{x\}[\sgq]\,,
$$
\ie  an element of the skew ring $\bC\{x\}[\sgq]$, where
$\bC\{x\}$ is the $\bC$-algebra of germs
of analytic function at zero and $\sgq f(x)=f(qx)\sgq$.
The associated
$q$-difference equations is
\beq\label{eq:q-diff}
\cL y(x)=a_\nu(x)y(q^\nu x)+a_{\nu-1}(x)y(q^{\nu-1}x)+\dots+a_0(x)y(x)=0\,.
\eeq
We suppose that $\a_\nu(x)\neq 0$, and we call
$\nu$ is the \emph{order} of $\cL$ (or of $\cL y=0$).

\begin{defn}
The \emph{Newton polygon} $NP(\cL)$ of the equation $\cL y=0$ (or of the operator $\cL$)
is the convex envelop in $\R^2$ of the following set:
$$
\l\{(i,k)\in\Z\times\R~:~i=0,\dots,\nu;~a_i(x)\neq 0,~k\geq\ord_x a_i(x)\r\}\,,
$$
where $\ord_x a_i(x)\geq 0$ denotes the order of zero of $a_i(x)$ at $x=0$.
\end{defn}

Notice that the polygon $NP(\cL)$ has a finite number of finite slopes,
which are all rational and can be negative, and two infinite vertical sides.
We will denote $\mu_1,\dots,\mu_k$ the finite slopes of $NP(\cL)$ (or, briefly of $\cL$),
ordered so that
$\mu_1<\mu_2<\dots<\mu_k$ (\ie from left to right),
and $r_1,\dots,r_k$ the length of their respective projections on the
$x$-axis. Notice that $\mu_ir_i\in\Z$ for any $i=1,\dots,k$.
\par
We can always assume, and we will actually assume, that
the boundary of the Newton polygon of $\cL$ and the $x$-axis
intersect only in one point or in a segment,
by clearing some common powers of $x$ in the coefficients of $\cL$.
Once this convention fixed, the Newton polygon is completely determined
by the set $\{(\mu_1,r_1),\dots,(\mu_k,r_k)\}\in\Q\times\Z_{>0}$, therefore
we will identify the two data.

\begin{defn}
A $q$-difference operator, whose Newton polygon has only one slope (equal to $\mu$)
is called \emph{pure (of slope $\mu$)}.
\end{defn}

\begin{rmk}\label{rmk:ManipSlopes}
All the properties of Newton polygons of $q$-difference equations
listed in \cite[\S1.1]{sauloyfiltration}
are formal and therefore independent of the field $\bC$ and of the norm of $q$: they can be rewritten,
with exactly the same proof, in our case.
We recall, in particular, two properties of the Newton polygon that we will use in the sequel
(\cf \cite[\S1.1.5]{sauloyfiltration}):\\
$\bullet$
Let $\theta$ be a solution, in some formal extension of
$\bC(\{x\})=\mathop{Frac}(\bC\{x\})$, of the $q$-difference equation $y(qx)=xy(x)$.
The twisted conjugate operator
$x^C\theta^\mu\cL\theta^{-\mu}\in\bC\{x\}[\sgq]$,
where $C$ is a convenient non negative integer,
is associated to the $q$-difference equation\footnote{Notice that there is
no need of determine the function $\theta$.}
\beq\label{eq:twistedconj}
a_\nu(x)q^{-\mu\frac{\nu(\nu+1)}{2}}x^{C-\mu\nu}y(q^\nu x)
+a_{\nu-1}(x)q^{-\mu\frac{\nu(\nu-1)}{2}}x^{C-\mu(\nu-1)}y(q^{\nu-1}x)
+\dots+x^C a_0(x)y(x)=0\,,
\eeq
and has Newton polygon $\{(\mu_1-\mu,r_1),\dots,(\mu_k-\mu,r_k)\}$.\\
$\bullet$
If $e_{q,c}(x)$ is a solution of $y(qx)=cy(x)$, with $\in\bC^\ast$.
Then the twisted operator
$e_{q,c}(x)^{-1}\cL e_{q,c}(x)$ has the same Newton polygon as $\cL$,
while all the zeros of the polynomial $\sum_{i=0}^\nu a_i(0)T^i$
are multiplied by $c$.
\end{rmk}

\subsection{Admissible $q$-difference operators}
\label{subsec:admissibleoperator}

Suppose that $0$ is a slope of $NP(\cL)$.
We call \emph{characteristic polynomial of the zero slope} the polynomial
$$
a_\nu(0)T^\nu+a_{\nu-1}(0)T^{\nu-1}+\dots+a_0(0)=0\,.
$$
The characteristic polynomial of a slope $\mu\in\Z$
is the characteristic polynomial of the zero slope of the $q$-difference operator
$x^C\theta^{\mu}\cL\theta^{-\mu}$ (\cf Equation \ref{eq:twistedconj}).
In the general case, when $\mu\in\Q\smallsetminus\Z$, we reduce to the previous assumption
by performing a ramification.
Namely, for a convenient $n\in\Z_{>0}$, we set $t=x^{1/n}$.
With this variable change, the operator $\cL$ becomes
$\sum a_i(t^n)\sg_{q^{1/n}}^i$.
Notice that the characteristic polynomial does not depend on
the choice of $n$.
\par
Finally, we call the non zero roots of the characteristic polynomial of the slope $\mu$ the
\emph{exponents of the slope $\mu$}. The cardinality of the set $\Exp(\cL,\mu)$ of the exponents of the
slope $\mu$, counted with multiplicities, is equal to the length of the
projection of $\mu$ on the $x$-axis.

\begin{defn}\label{defn:admissible}
Let $(\la_1,\dots,\la_r)$ be the exponents of the slope $\mu$ of $\cL$ and
let
$$
\ul\La=\{\la_{i}\la_{j}^{-1}~:~ i,j=1,\dots,r\,;\,\,\la_{i}\la_{j}^{-1}\not\in q^{\Z_{\leq 0}}\}\,.
$$
We say that a slope $\mu\in\Z$ of $\cL$ is \emph{admissible} if the series
$\phi_{(q;\ul\La)}(x)$ (\cf Eq. \ref{eq:smalldivisors})
is convergent and that a slope $\mu\in\Q$
is \emph{almost admissible} if it becomes
admissible in $\bC\{x^{1/n}\}[\sg_{q^{1/n}}]$, for a convenient $n\in\Z_{>0}$.
\par
A $q$-difference operator is said to be \emph{admissible}
(resp. \emph{almost admissible}) if
all its slopes are admissible (resp. \emph{almost admissible}).
\end{defn}

\begin{rmk}
A rank 1 $q$-difference equation is admissible as soon as the series
$\sum_{n\geq 0}\frac{x^n}{(q;q)_n}$ is convergent.
\end{rmk}

\subsection{Analytic factorization of admissible $q$-difference operators}

The main result of this subsection is the analytic factorization
of admissible $q$-difference operators.
The analogous result
in the case $|q|\neq 1$ is well known
(\cf \cite{MarotteZhang}, \cite[\S1.2]{sauloyfiltration},
or, for a more detailed exposition, \cite[\S1.2]{Spoligono}). The germs of those
works are already in \cite{BirkhoffGuenther}, where the authors establish
a canonical form for solution of analytic $q$-difference systems.

\begin{thm}\label{thm:adams}
Suppose that the $q$-difference operator $\cL$ is admissible, with Newton polygon
$\{(\mu_1,r_1),\dots,(\mu_k,r_k)\}$.
Then for any permutation $\varpi$ of the set $\{1,\dots,k\}$
there exists a factorization of $\cL$:
$$
\cL=\cL_{\varpi,1}\circ\cL_{\varpi,2}\circ\dots\circ\cL_{\varpi,k}\,,
$$
such that $\cL_{\varpi,i}\in\bC\{x\}[\sgq]$ is
admissible and pure of slope $\mu_{\varpi(i)}$ and order $r_{\varpi(i)}$.
\end{thm}

\begin{rmk}~\\
$\bullet$
Given the permutation $\varpi$, the $q$-difference operator $\cL_{\varpi,i}$
is uniquely determined, modulo a factor in  $\bC\{x\}$.\\
\smallskip
$\bullet$
Exactly the same statement holds for almost admissible $q$-difference operator
(\cf Theorem \ref{thm:AdamsAlmostAdmissible} below).
\end{rmk}

Theorem \ref{thm:adams} follows from the recursive application of the statement:

\begin{prop}\label{prop:adams}
Let $\mu\in\Z$ be an admissible slope of the Newton polygon of $\cL$ and let $r$ be the
length of its projection on the $x$-axis.
Then the $q$-difference operator $\cL$ admits a factorization
$\cL=\wtilde\cL\circ\cL_\mu$,
such that:
\begin{trivlist}
\item 1.
the operator $\wtilde\cL$ is in $\bC\{x\}[\sgq]$ and
$NP(\wtilde\cL)=NP(\cL)\smallsetminus\{(\mu,r)\}$;
\item 2.
the operator $\cL_\mu$ has the form:
$$\cL_\mu=(x^\mu\sgq-\la_r)h_r(x)\circ(x^\mu\sgq-\la_{r-1})h_{r-1}(x)\circ
\cdots\circ (x^\mu\sgq-\la_1)h_1(x)\,,$$
where:
\begin{itemize}
\item
$\la_1,\dots,\la_r\in\bC$ are the exponents of the slope $\mu$,
ordered so that if $\frac{\la_i}{\la_j}\in q^{\Z_{>0}}$ then $i<j$;
\item
$h_1(x),\dots,h_r(x)\in 1+x\bC\{x\}$.
\end{itemize}
\end{trivlist}
Moreover if $\cL$ is admissible (resp. almost admissible),
the operator $\wtilde\cL$ is also admissible (resp. almost admissible).
\end{prop}

Proposition \ref{prop:adams} itself follows from an iterated
application of the lemma below:

\begin{lemma}\label{lemma:adams}
Let $(\mu,r)\in NP(\cL)=\{(\mu_1,r_1),\dots,(\mu_k,r_k)\}$ be an integral slope of $\cL$
with exponents $(\la_1,\dots,\la_r)$.
Fix an exponent $\la$ of $\mu$ such that:\\
\smallskip
1. $q^n\la$ is not an exponent of the same slope
for any $n>0$;\\
\smallskip
2.
the series $\phi_{\l(q;\l(\frac{\la_1}{\la},\dots,\frac{\la_r}{\la}\r)\r)}(x)$ is convergent.\\
\smallskip
Then there exists a unique $h(x)\in1+x\bC\{x\}$ such that
$\cL=\wtilde\cL\circ(x^\mu\sgq-\la)h(x)$, for some $\wtilde\cL\in\bC\{x\}[\sgq]$.
Moreover let $\iota=1,\dots,k$ such that $\mu_\iota=\mu$:\\
\smallskip
$\bullet$
if $r_\iota=1$ then
$NP(\wtilde\cL)=\{(\mu_1,r_1),\dots,(\mu_{\iota-1},r_{\iota-1}),
(\mu_{\iota+1},r_{\iota+1}),\dots,(\mu_k,r_k)\}$;\\
\smallskip
$\bullet$
if $r_\iota>1$ then
$NP(\wtilde\cL)=\{(\mu_1,r_1),\dots,(\mu_\iota,r_\iota-1),\dots,(\mu_k,r_k)\}$
and $\Exp(\wtilde\cL,\mu_\iota)=\Exp(\cL,\mu_\iota)\smallsetminus\{\la\}$.
\end{lemma}

\begin{proof}
It is enough to prove the lemma for $\mu=0$ and $\la=1$ (\cf Remark \ref{rmk:ManipSlopes}).
Write $y(x)=\sum_{n\geq 0}y_nx^n$, with $y_0=1$, and $a_i(x)=\sum_{n\geq 0}a_{i,n}x^n$.
Then we obtain by direct computation that $\cL y(x)=0$ if and only if for any $n\geq 1$
we have:
$$
F_0(q^n)y_n=-\sum_{l=0}^{n-1}F_{n-l}(q^l)y_l\,,
$$
where
$F_l(T)=\sum_{i=0}^\nu a_{i,l}T^i$. Remark that assumption 1 is equivalent to the property:
$F_0(q^n)\neq 0$ for any $n\in\Z_{>0}$.
\par
The convergence of the coefficients $a_i(x)$ of $\cL$
implies the existence of two constants $A,B>0$ such that $|F_{n-l}(q^l)|\leq AB^{n-l}$,
for any $n\geq 0$ and any $l=0,\dots,n-1$.
We set
$$
s_n=F_0(1)F_0(q)\dots F_0(q^n)y_n\,.
$$
Then
$$
|s_n|
\leq\l|\sum_{l=0}^{n-1}s_lF_0(q^{l+1})\cdots F_0(q^{n-1})F_{n-l}(q^l)\r|
\leq A^nB^n\sum_{l=0}^{n-1}\frac{|s_l|}{(AB)^l}\,,
$$
and therefore:
$$
|t_n|\leq\sum_{l=0}^{n-1}|t_l|\,,
\hbox{~with~}t_l=\frac{s_l}{(AB)^l}\,.
$$
If $|t_l|<CD^l$, for any $l=0,\dots,n-1$, with $D>1$,
then $|t_n|\leq C\sum_{l=0}^{n-1}D^l\leq CD^n(D-1)^{-1}\leq CD^n$.
Therefore $|t_n|\leq CD^n$ for any $n\geq 1$, and hence  $|s_n|\leq C(ABD)^n$.
Hypothesis 2 assures that the series
$\sum_{n\geq 1}\frac{x^n}{F_0(1)\cdots F_0(q^n)}$ is convergent and therefore
that $y(x)$ is convergent. We conclude setting $h(x)=y(x)^{-1}$.
\par
For the assertion on the Newton polygon \cf\cite{sauloyfiltration}.
\end{proof}

For further reference we point out that we have actually proved the
following corollaries:

\begin{cor}\label{cor:adams}
Under the hypothesis of Lemma \ref{lemma:adams}, suppose that $\cL$ has a right factor
of the form $(\sgq^\mu-\la)\circ h(x)$, with $\mu\in\Q$, $\la\in\bC^\ast$ and $h(x)\in \bC[[x]]$.
Then $h(x)$ is convergent.
\end{cor}

\begin{rmk}
Corollary \ref{cor:adams} above generalizes \cite[Thm. 6.1]{Bezivin},
where the author proves that a formal solution of an analytic $q$-difference operator
satisfying some diophantine assumptions is always convergent.
\end{rmk}

\begin{cor}\label{cor:adamsKn}
Any almost admissible $q$-difference operator $\cL$ admits an analytic
factorization in $\bC\{x^{1/n}\}[\sgq]$,
with $\sgq x^{1/n}=q^{1/n}x^{1/n}$, for a convenient
$n\in\Z_{>0}$.
\par
The irreducible
factors of $\cL$ in $\bC\{x^{1/n}\}[\sgq]$ are of the
form $(x^{\mu/n}\sgq-\la)h(x^{1/n})$, with $\mu\in\Z$, $\la\in\bC^\ast$ and
$h(x^{1/n})\in 1+x^{1/n}\bC\{x^{1/n}\}$.
\end{cor}

The following example shows the importance
of considering admissible operators.

\begin{exa}\label{exa:admissibility}
The series $\Phi(x)=\Phi_{(q;q\la)}((1-q)x)$, studied in Proposition \ref{prop:smalldivisors},
is solution of the $q$-difference operator
$\cL=(\sgq-1)\circ[\la\sgq-((q-1)x+1)]$. This operator is already factored.
\par
Suppose that $\la\not\in q^{\Z_{<0}}$.
If the series $\Phi(x)$ is convergent, \ie if $\cL$ is admissible, the operator
$(\sgq-1)\circ\Phi(x)^{-1}$ is a right factor of $\cL$, as we could have deduced from
Lemma \ref{lemma:adams}. We conclude that if $\Phi(x)$ is not convergent the operator $\cL$ cannot
be factored ``starting with the exponents 1''.
\end{exa}

\subsection{A digression on formal factorization of $q$-difference operators}

If we drop the diophantine assumption of admissibility and consider
an operator $\cL\in\bC[[x]][\sgq]$, the notions of Newton polygon and exponent still make sense.
The following result is well known (\cf \cite{SoibelmanVologodsky}, \cite{sauloyfiltration})
and can be proved reasoning as in the previous section:

\begin{thm}\label{thm:adamsformale}
Suppose that the $q$-difference operator $\cL\in\bC[[x]][\sgq]$ has Newton polygon
$\{(\mu_1,r_1),\dots,(\mu_k,r_k)\}$, with integral slopes.
Then for any permutation $\varpi$ of the set $\{1,\dots,k\}$
there exists a factorization of $\cL$:
$$
\cL=\cL_{\varpi,1}\circ\cL_{\varpi,2}\circ\dots\circ\cL_{\varpi,k}\,,
$$
such that $\cL_{\varpi,i}\in\bC[[x]][\sgq]$ is
pure of slope $\mu_{\varpi(i)}$ and order $r_{\varpi(i)}$.
Any $\cL_{\varpi,i}$ admits a factorization of the form
$$
\cL_{\varpi,i}
=(x^{\mu_{\varpi(i)}}\sgq-\la_{r_{\varpi(i)}})h_{r_{\varpi(i)}}(x)\circ
(x^{\mu_{\varpi(i)}}\sgq-\la_{r-1})h_{r-1}(x)\circ
\cdots\circ (x^{\mu_{\varpi(i)}}\sgq-\la_1)h_1(x)\,,
$$
where:
\begin{itemize}
\item
$\Exp(\cL,\mu_{\varpi(i)})=(\la_1,\dots,\la_{r_{\varpi(i)}})$ are the exponents of the slope $\mu_{\varpi(i)}$,
ordered so that if $\frac{\la_i}{\la_j}\in q^{\Z_{>0}}$ then $i<j$;
\item
$h_1(x),\dots,h_{r_{\varpi(i)}}(x)\in 1+x\bC[[x]]$.
\end{itemize}
\end{thm}

\section{Analytic classification of $q$-difference modules}

Let $\bK=\bC(\{x\})$ be the field of germs of meromorphic function at $0$,
\ie the field of fractions of $\bC\{x\}$.
In the following we will denote by $\what \bK=\bC((x))$ the field of Laurent series,
and by $\bK_n=\bK(x^{1/n})$ (resp. $\what \bK_n=\what\bK(x^{1/n})$)
the finite extension of $\bf K$ (resp. $\what \bK$) or degree $n$,
with its natural $q^{1/n}$-difference structure. \emph{We remind that
we are assuming all over the paper that the series
$\sum_{n\geq 0}\frac{x^n}{(q;q)_n}$ is convergent}.

\subsection{Generalities on $q$-difference modules}

We recall some generalities on $q$-difference modules
(for a more detailed exposition \cf for instance \cite[Part I]{DVInv},
\cite{sauloyfiltration} or \cite{gazette}).

\medskip
Let $F$ be a $q$-difference field over $\bC$, \ie a field $F/\bC$ of functions with
an action of $\sgq$.

\begin{defn}
A \emph{$q$-difference modules $\cM=(M,\Sgq)$ over $F$ (of rank $\nu$)}
is a finite $F$-vector space
$M$, of dimension $\nu$, equipped with a $\sgq$-linear bijective
endomorphism $\Sgq$, \ie with a $\bC$-linear isomorphism such that
$\Sgq(fm)=\sgq(f)\Sgq(m)$, for any $f\in F$ and any $m\in M$.
\par
A \emph{morphism of $q$-difference modules} $\varphi:(M,\Sgq^M)\to (N,\Sgq^N)$ is a
$\bC$-linear morphism $M\to N$, commuting to the action of $\Sgq^M$ and $\Sgq^N$,
\ie $\Sgq^N\circ\varphi=\varphi\circ\Sgq^M$.
\end{defn}

If $G$ is a $q$-difference field extending $F$ (\ie $G/F$ and the action of $\sgq$ on
$G$ extends the one on $F$),
the module $\cM_G=(M\otimes_F G,\Sgq\otimes\sgq)$ is
naturally a $q$-difference module over $G$.
\par
If $F_n$, $n\in\Z_{>1}$, is a $q^{1/n}$-difference field containing $F$ and such that
${\sg_{q^{1/n}}}_{\vert F}=\sgq$ (for instance, think of $\bK$ and $\bK_n$),
to any $q$-difference modules $\cM=(M,\Sgq)$ over $F$
we can associate the $q^{1/n}$-difference module
$\cM_{F_n}=(M\otimes_F F_n,\Sgq\otimes\sg_{q^{1/n}})$ over $F_n$.
\par
For other algebraic constructions (tensor product, internal $Hom$,...) we refer
to \cite{DVInv} or \cite{sauloyfiltration}.

\begin{rmk}[The cyclic vector lemma]~\\
The cyclic vector lemma says that a $q$-difference module $\cM$ over $F$, of rank $\nu$,
contains a cyclic element $m\in M$, \ie an element such that
$m,\Sgq m,\dots,\Sgq^{\nu-1}m$ is an $F$-basis of $M$. This is equivalent to say
that there exists a
$q$-difference operator $\cL\in F[\sgq]$ of order $\nu$ such that
we have an isomorphism of $q$-difference modules
$$
M\cong \frac{F[\sgq,\sgq^{-1}]}{F[\sgq,\sgq^{-1}]\cL}\,.
$$
We will call $\cL$ a $q$-difference operator associated to $\cM$, and
$\cM$ the $q$-difference module associated to $\cL$.
\end{rmk}

\begin{exa}
\label{exa:rank1}
(Rank 1 $q$-difference modules\footnote{For more details on the
rank one case \cf\cite[Prop. 3.6]{SoibelmanVologodsky}, where the Picard group of $q$-difference
modules modules over $\cO(\C^\ast)$, satisfying a convenient diophantine assumption, is studied.})\\
Let $\mu\in\Z$, $\la\in\bC^\ast$ and $h(x)\in\bK$ (resp. $h(x)\in\what\bK$).
Let us consider the rank 1 $q$-difference module
$\cM_{\mu,\la}=(M_{\mu,\la},\Sgq)$ over $\bK$ (resp. $\what\bK$)
associated to the operator $(x^\mu\sgq-\la)\circ h(x)=h(qx)x^\mu\sgq-h(x)\la$.
This means that there exists a basis $f$ of $M_{\mu,\la}$ such that
$\Sgq f=\frac{h(x)}{h(qx)}\frac \la{x^\mu}f$.
If one consider the basis $e=h(x)f$, then $\Sgq e=\frac \la{x^\mu}e$.
\par
A straightforward calculation shows that
$\cM_{\mu,\la}$ is isomorphic, as a
$q$-difference module, to $\cM_{\mu^\p,\la^\p}$ if and only if $\mu=\mu^\p$ and
$\frac \la{\la^\p}\in q^\Z$.
Moreover, we have proved in the previous section that a $q$-difference operator
$\sgq-a(x)$, with $a(x)\in\what\bK$ can be always be written
in the form $\sgq-\frac\la{x^\mu}\frac{h(x)}{h(qx)}$,
for some $h(x)\in\what\bK$. We also know that,
if $q$ is such that $\sum_{n\geq 0}\frac{x^n}{(q;q)_n}$ converges and if $a(x)\in\bK$,
then $h(x)$ is a convergent series.
\end{exa}

The remark and the example above, together with the results of the previous section,
imply that we can attach to a $q$-difference modules a \emph{Newton polygon} by choosing a
cyclic vector and that the Newton polygon of a $q$-difference modules is well-defined
(\cf \cite{sauloyfiltration}). Moreover the classes modulo $q^\Z$ of the exponents of each slope
are independent of the choice of the cyclic vector
(\cf \cite[Thm. 3.12 and 3.14]{SoibelmanVologodsky} and \cite{sauloyfiltration}).
Both the Newton polygon and the classes modulo $q^\Z$ of the exponents are an invariant
of the formal isomorphism class.

\subsection{Main result}

Let us call $\cB_q$ (resp. $\hat\cB_q$)
the category of $q$-difference module over $\bK$
(resp. $\what\bK$).
We will use the adjective analytic (resp. formal) to refer to objects,
morphisms, isomorphism classes, etc etc of $\cB_q$ (resp. $\hat\cB_q$).
\par
We are concerned with the problem of
finding the largest full subcategory $\cB_q^{iso}$
of $\cB_q$ defined by the following property:
\begin{quote}
\emph{An object $\cM$ of $\cB_q$ belongs to $\cB_q^{iso}$ if any object $\cN$ in $\cB_q$
such that $\cN_{\what\bK}$ is isomorphic to $\cM_{\what\bK}$ in $\what\cB_q$
is already isomorphic to $\cM$ in $\cB_q$.}
\end{quote}
This means that restriction of the functor
\beq\label{eq:functor}
\begin{array}{rccc}
-\otimes_\bK\what\bK:
&\cB_q&\longrightarrow&\what\cB_q\\
&\cM&\longmapsto&\cM_{\what\bK}
\end{array}
\eeq
to $\cB_q^{iso}$ is an equivalence of category between $\cB_q^{iso}$ and its image.
We will come back in \S\ref{sec:Soibelman} to the characterization of
$\cB_q^{iso}\otimes_\bK\what\bK$ inside $\what\cB_q$.
A counterexample of the fact that
the functor $-\otimes_\bK\what\bK$ is not an equivalence of category in general
is considered in \S\ref{subsec:example}.
\par
The category $\cB_q^{iso}$ is link to the notion of admissibility introduced in the
previous section:

\begin{defn}
We say that a $q$-difference module $\cM$ over $\bK$ is
\emph{admissible (resp. almost admissible; resp. pure (of slope $\mu$))}
if there exists an operator $\cL\in\bC\{x\}[\sgq]$ such that $M\cong \bK[\sgq]/(\cL)$
and that $\cL$
is admissible (resp. almost admissible; resp. pure (of slope $\mu$)).
\end{defn}

\begin{rmk}
The considerations in the previous section imply that the notion of (almost) admissible
$q$-difference module is well defined and invariant up to isomorphism.
\end{rmk}

Our main result is:

\begin{thm}\label{thm:mainthm}
The category $\cB_q^{iso}$ is the full subcategory of $\cB_q$ whose objects are
almost admissible $q$-difference modules.
\end{thm}

We introduce some notation that will be useful in the proof of Theorem \ref{thm:mainthm}.
We will denote \qdiffKa\ (resp. \qdiffKaa)
the category of admissible
(resp. almost admissible) $q$-difference modules over $\bK$, whose objects are
the admissible
(resp. almost admissible) $q$-difference
modules over $\bK$ and whose morphisms are the morphisms of
$q$-difference modules over $\bK$.

\begin{rmk}
We know that $\cB_q$ and $\what\cB_q$ are abelian categories. Therefore, $\ker$ and $\hbox{coker}$
of morphisms in \qdiffKa\ (resp. \qdiffKaa) are $q$-difference modules over $\bK$.
To prove that they are objects of \qdiffKa\ (resp. \qdiffKaa) we have only to
point out that the operator associated to a sub-$q$-difference module
(resp. a quotient module) is a right (resp. left) factor
of a convenient operator associated to the module itself, in fact the slopes and
the classes modulo $q^\Z$ of the exponents associated to each slope are invariants of
$q$-difference modules.
\end{rmk}

The proof of Theorem \ref{thm:mainthm} consists in proving that $\cB_q^{iso}=$\qdiffKaa\
and is articulated in the following steps:
first of all we will make a list of simple and indecomposable objects of
\qdiffKaa; then we prove a structure theorem for almost admissible $q$-difference modules.
We deduce that the formal isomorphism class of an object of $\cB_q$ correspond to
more than one analytic isomorphism class if and only if the slopes of the
Newton polygon are not admissible: this means that $\cB_q^{iso}$ and \qdiffKaa\ coincide.

\subsection{A crucial example}
\label{subsec:example}

Consider the $q$-difference operator (\cf Example \ref{exa:admissibility})
$$
\cL=(\sgq-1)\circ[\la\sgq-((q-1)x+1)]
$$
and its associated
$q$-difference module $\cM=\l(M=\frac{\bK[\sgq,\sgq^{-1}]}{\bK[\sgq,\sgq^{-1}]\cL},\Sgq\r)$.
If $\la\in q^\Z$ the module is admissible and there is nothing more to prove.
So let us suppose that $\la\not\in q^\Z$.
\par
In $\what\cB_q$, the $q$-difference module $\what\cM=\cM_{\what\bK}$
is isomorphic to the rank 2 module
$\what\bK^2$ equipped with the semi-linear operator:
$$
\begin{array}{rccc}
\Sgq:&\what\bK^2&\longrightarrow&\what\bK^2\\
&\begin{pmatrix}f_1(x)\\f_2(x)\end{pmatrix}&\longmapsto &
    \begin{pmatrix}1&0\\0&\la^{-1}\end{pmatrix}
    \begin{pmatrix}f_1(qx)\\f_2(qx)\end{pmatrix}
\end{array}\,.
$$
In fact, $\cL$ has a right factor $\la\sgq-((q-1)x+1)$: this corresponds
to the existence of an element $f\in M$ such that
$\Sgq f=\la^{-1}\l((q-1)x+1\r) f$.
Since
$e_q(x)=\sum_{n\geq 0}\frac{(1-q)^nx^n}{(q;q)_n}$ is solution of the
equation $y(qx)=\l((q-1)x+1\r)y(x)$, we deduce that $\wtilde f=e_q(x)f$ verifies
$\Sgq\wtilde f=\la^{-1}\wtilde f$.
On the other hand, we have seen that there always exists $\Phi\in\bC[[x]]$ such that
$\l(\sgq-1\r)\Phi$ is a right factor of $\cL$: this means that there exists
$e\in M$ such that $\Sgq e=\Phi(x)\Phi(qx)^{-1}e$ and therefore that there exists
$\wtilde e\in M_{\what\bK}$ such that $\Sgq\wtilde e=\wtilde e$. A priori this
last base change is only formal: the series $\Phi$ converges if and only if the
module is admissible; \cf Example \ref{exa:admissibility}.
\par
The calculations above say more: the formal isomorphism class
of $M$ corresponds to a single analytic isomorphism class if and only if $M$ is admissible,
which happens if and only if the series $\sum_{n\geq 0}\frac{x^n}{(q;q)_n}$
and $\sum_{n\geq 0, q^n\neq\la}\frac{x^n}{q^n-\la}$ converge.

\subsection{Simple and indecomposable objects}
\label{subsec:simpleobjects}

In differential and difference equation theory, simple objects are called \emph{irreducible}.
They are those objects $\cM=(M,\Sgq)$ over $\bK$ such that any $m\in M$
is a cyclic vector: this is
equivalent to the property of not having a proper $q$-difference sub-module, or to the
fact that any $q$-difference operator associated to $\cM$ cannot be factorized in $\bK[\sgq]$.

\begin{cor}
The only irreducible objects in the category \qdiffKa are
the rank one modules described in Example \ref{exa:rank1}.
\end{cor}

\begin{proof}
It is a consequence of Proposition \ref{prop:adams}.
\end{proof}

Before describing the irreducible object of the category
\qdiffKaa, we need to introduce a functor of
restriction of scalars going from \qdiffKna\ to
\qdiffKaa.
In fact, the set $\{1,x^{1/n},\dots,x^{n-1/n}\}$ is a basis of $\bK_n/\bK$
such that $\sgq x^{i/n}=q^{i/n}x^{i/n}$.
Therefore $\bK_n$ can be identified to the admissible $q$-difference module
$M_{0,1}\oplus M_{0,q^{1/n}}\oplus\dots\oplus M_{0,q^{n-1/n}}$
(in the notation of Example \ref{exa:rank1}).
\par
In the same way, we can associate to any (almost) admissible $q^{1/n}$-difference
module $\cM$ of rank $\nu$ over $\bK_n$ an almost admissible difference module $Res_n(\cM)$
of rank $n\nu$ over $\bK$ by restriction of scalars.
The functor $Res_n$ ``stretches'' the Newton polygon horizontally, meaning that
if the Newton polygon of $\cM$ over $\bK_n$ is $\{(\mu_1,r_1),\dots,(\mu_k,r_k)\}$, then the
Newton polygon of $Res_n(\cM)$ over $\bK$ is $\{(\mu_1/n,nr_1),\dots,(\mu_k/n,nr_k)\}$

\begin{exa}
Consider the $q^{1/2}$-module over $\bK_2$ associated to the equation
$x^{1/2}y(qx)=\la y(x)$, for some $\la\in\bC^\ast$. This means that we consider a
rank 1 module $\bK_2 e$ over $\bK_2$, such that
$\Sgq e=\frac\la{x^{1/2}}e$. Notice that its Newton polygon over $\bK_2$ has only one
single slope equal to 1.
Since $\bK_2 e=\bK e+\bK x^{1/2}e$, the module
$\bK_2e$ is a $q$-difference module of rank 2 over $\bK$, whose $q$-difference
structure is defined by:
$$
\Sgq(e,x^{1/2}e)=(e,x^{1/2}e)
\begin{pmatrix}
0 & q^{1/2}\la\\
\la/x & 0
\end{pmatrix}\,.
$$
Consider the vector $m=e+x^{1/2}e$. We have:
$\Sgq(m)=q^{1/2}\la e+\frac\la x(x^{1/2}e)$ and
$\Sgq^2(m)=\frac{q^{1/2}\la^2}{qx}e+\frac{q^{1/2}\la^2}{x}(x^{1/2}e)$.
Since $m$ and $\Sgq(m)$ are linearly
independent, $m$ is a cyclic vector for $\bK_2 e$ over $\bK$.
Moreover, for
$$
\l\{\begin{array}{l}
P(x)=-\la^2(q^{3/2}x-1)\\
Q(x)=\la(q-1)x\\
R(x)=-q^{1/2}x(q^{1/2}x-1)
\end{array}\r.
$$
we have $P(x)m+Q(x)\Sgq(m)=R(x)\Sgq^2(m)$.
In other words,
the Newton polygon of the rank 2
$q$-difference module $\bK_2e$ over $\bK$ has only one slope equal to 1/2.
\end{exa}

Let $n\in\Z_{>0}$, $\mu$ be an integer prime to $n$ and
$\cM_{\mu,\la,n}$ be the rank one module over $\bK_n$ associated to the equation
$x^{\mu/n}y(qx)=\la y(x)$.
In \cite[Lemma 3.9]{SoibelmanVologodsky}, Soibelman and Vologodsky
show that $\cN_{\mu/n,\la}=Res_n(\cM_{\mu,\la,n})$ is a simple object over $\cO(\C^\ast)$.
We show that all the simple objects of
the category \qdiffKaa\ are of this form
(for the case $|q|\neq 1$, \cf \cite{vdPutReversatToulouse}).
Remark that $\cM_{\mu,\la}=\cM_{\mu,\la,1}=\cN_{\mu,\la}$ as
$q$-difference modules over $\bK$.
\par
Let us start by proving the lemma:

\begin{lemma}\label{lemma:fattorizzazioneirriducile}
Let $\cM$ be a $q$-difference module associated to a
$q$-difference operator $\cL\in\bC\{x\}[\sgq]$. Suppose that the
operator $\cL$ has a right factor in $\bC\{x^{1/n}\}[\sgq]$ of the form
$(x^{\mu/n}\sgq-\la)\circ h(x^{1/n})$, with $n\in\Z_{>1}$,
$\mu\in\Z$, $(n,\mu)=1$, $\la\in\bC^\ast$ and $h(x)\in\bC\{x^{1/n}\}$.
\par
Then $\cM$ has a submodule isomorphic to $\cN_{\mu/n,\la}$.
\end{lemma}

\begin{proof}
First of all remark that any operator $\cL\in\bC\{x\}[\sgq]$
divisible by $(x^{\mu/n}\sgq-\la)\circ h(x^{1/n})$ has order $\geq n$.
\par
Let $\cL_{\mu/n,\la}\in\bC\{x\}[\sgq]$ be a $q$-difference operator (of order $n$)
associated to $\cN_{\mu/n,\la}$.
Since the ring $\bC\{x\}[\sgq]$ is euclidean the exist $\cQ,\cR\in\bC\{x\}[\sgq]$,
such that
$$
\cL=\cQ\circ\cL_{\mu/n,\la}+\cR\,,
$$
with $\cR=0$ or $\cR$ of order strictly smaller than $n$ and divisible on the right by
$(x^{\mu/n}\sgq-\la)\circ h(x^{1/n})$. Of course, if $\cR\neq 0$, we obtain a contradiction.
Therefore $\cL_{\mu/n,\la}$ divides $\cL$ and the lemma follows.
\end{proof}

\begin{rmk}
The same statement holds for formal operator $\cL\in\bC[[x]][\sgq]$, having a formal right
factor $(x^{\mu/n}\sgq-\la)\circ h(x^{1/n})$, with $h(x)\in\bC[[x^{1/n}]]$.
\end{rmk}

Finally we have a complete description of the isomorphism classes
of almost admissible irreducible $q$-difference modules over $\bK$:

\begin{prop}
A system of representatives of the isomorphism classes of the irreducible
objects of \qdiffKaa\ (resp. $\what\cB_q$) is given by the reunion of the following sets:\\
- rank 1 $q$-difference modules $\cM_{\mu,\la}$, with $\mu\in\Z$ and
$c\in\bC^\ast/q^\Z$, \ie the irreducible objects of \qdiffKa\
up to isomorphism (\cf Example \ref{exa:rank1});\\
- the $q$-difference modules $\cN_{\mu/n,\la}=Res_n(\cM_{\mu,\la,n})$, where
$n\in\Z_{>0}$, $\mu\in\Z$, $(n,\mu)=1$, and $\la\in\bC^\ast/(q^{1/n})^\Z$.
\end{prop}

\begin{proof}
The corollary is well known for $\what\cB_q$. Let us prove the statement for the
category \qdiffKaa.
\par
Rank 1 irreducible objects of \qdiffKaa\ are necessarily admissible, therefore
they are of the form $\cM_{\mu,\la}$, for some $\mu\in\Z$ and some $\la\in\bC^\ast/q^\Z$.
Consider an irreducible object $\cM$ in \qdiffKaa\ of higher rank. Because of the
previous lemma and of Corollary \ref{cor:adamsKn},
it must contain an object of the form $N_{\mu,\la,n}$, for convenient
$\mu$, $\la$, $n$. The irreducibility implies that $\cM\cong\cN_{\mu,\la,n}$.
\end{proof}

\begin{rmk}
Consider the rank 1 modules $\cN_{\mu,\la,n}$ over $\bK_n$ and $\cN_{r\mu,\la,rn}$
over $\bK_{rn}$, for some $\mu,r,n\in\Z$, $r>1$, $n>0$, $(\mu,n)=1$, and $\la\in\C^\ast$.
Then $Res_n(\cN_{\mu,\la,n})$ is a rank $n$ $q$-difference module over $\bK$, while
$Res_{rn}(\cN_{r\mu,\la,rn})$ has rank $rn$, although $\cN_{\mu,\la,n}$ and $\cN_{r\mu,\la,rn}$ are
associated to the same rank one operator.
\par
Writing explicitly the basis of $\bK_{rn}$ over $\bK_n$ and over $\bK$,
one can show that $Res_{rn}(\cN_{r\mu,\la,rn})$ is a direct sum of $r$ copies of
$Res_n(\cN_{\mu,\la,n})$.
\end{rmk}

\subsection{Structure theorem for almost admissible
$q$-difference mdoules}

Now we are ready to state a structure theorem for almost admissible
$q$-difference modules:

\begin{thm}\label{thm:M-adams}
Suppose that the $q$-difference module $\cM=(M,\Sgq)$ over $\bK$
is almost admissible, with Newton polygon
$\{(\mu_1,r_1),\dots,(\mu_k,r_k)\}$.
Then
$$
M=M_1\oplus M_2\oplus\dots\oplus M_k\,,
$$
where the $q$-difference modules $\cM_i=(M_i,{\Sgq}_{\vert {M_i}})$ are defined over $\bK$,
almost admissible and pure of slope $\mu_i$ and rank $r_i$.
\par
Each $\cM_i$ is a direct sum of
almost admissible indecomposable
$q$-difference modules, \ie iterated non trivial extension
of a simple almost admissible $q$-difference module by itself.
\end{thm}

\begin{rmk}\label{rmk:indecomposable}
More precisely, consider the rank $\nu$
unipotent $q$-difference module $\mathcal U_\nu=(U_\nu,\Sgq)$, defined by the
property of having a basis $\ul e$ such that the action of $\Sgq$ on $\ul e$ is described by
a matrix composed by a single Jordan block with eigenvalue 1.
Then the indecomposable modules $\cN$ in the previous theorem are isomorphic
to $\cN\otimes_\bK{\mathcal U}_{\nu}$, for some irreducible module $\cN$
of \qdiffKaa\ and some $\nu$.
\end{rmk}

The theorem above is equivalent to a stronger version of
Theorem \ref{thm:adams} for almost admissible
$q$-difference operators:

\begin{thm}\label{thm:AdamsAlmostAdmissible}
Suppose that the $q$-difference operator $\cL$ is almost admissible, with Newton polygon
$\{(\mu_1,r_1),\dots,(\mu_k,r_k)\}$.
Then for any permutation $\varpi$ on the set $\{1,\dots,k\}$
there exists a factorization of $\cL$:
$$
\cL=\cL_{\varpi,1}\circ\cL_{\varpi,2}\circ\dots\circ\cL_{\varpi,k}\,,
$$
such that $\cL_{\varpi,i}\in\bC\{x\}[\sgq]$ is almost
admissible and pure of slope $\mu_{\varpi(i)}$ and order $r_{\varpi(i)}$.
\par
Moreover, for any $i=1,\dots,k$, write $\mu_i=d_i/s_i$,
with $d_i,s_i\in\Z$,
$s_i>0$ and $(d_i,s_i)=1$.
We have:
$$
\cL_{\varpi,i}=\cL_{d_{\varpi(i)},\la_l^{\varpi(i)},s_{\varpi(i)}}\circ\dots
\circ\cL_{d_{\varpi(i)},\la_1^{\varpi(i)},s_{\varpi(i)}}\,,
$$
where:\\
$\bullet$
$\la_1^{\varpi(i)},\dots,\la_l^{\varpi(i)}$
are exponents of the slope $\mu_{\varpi(i)}$, ordered so that
$\la^{\varpi(i)}_j\l(\la^{\varpi(i)}_{j^\p}\r)^{-1}\in q^{\Z_{>0}}$ then $j<j^\p$;\\
$\bullet$
the operator $\cL_{d_{\varpi(i)},\la_j^{\varpi(i)},s_{\varpi(i)}}$ is associated to the module
$\cN_{d_{\varpi(i)},\la_j^{\varpi(i)},s_{\varpi(i)}}$.\\
\end{thm}

\begin{proof}
Suppose that the operator has at least one non integral slope.
\emph{A priori} the operators $\cL_{\varpi,i}$ are defined over $\bC\{x^{1/n}\}$,
for some $n>1$. But it follows from
Lemma \ref{lemma:fattorizzazioneirriducile} that they are product of
operators associated to $q$-difference modules defined over $\bK$, of the form
$\cN_{\mu,\la,n}$, for same $\mu,n\in\Z$, $n>0$, and $\la\in\bC^\ast$.
\end{proof}

\subsection{Analytic vs formal classification}

The formal classification of $q$-difference modules with $|q|=1$
is studied in \cite{SoibelmanVologodsky}, by different techniques.
It can also be deduced by the results of the previous section, dropping the
diophantine assumptions, and establishing a formal factorization theorem for $q$-difference
operators:

\begin{thm}\label{thm:M-formaladams}
Consider a $q$-difference module $\cM=(M,\Sgq)$ over $\what \bK$, with
Newton polygon
$\{(\mu_1,r_1),\dots,(\mu_k,r_k)\}$.
Then
$$
M=M_1\oplus M_2\oplus\dots\oplus M_k\,,
$$
where the $q$-difference modules $\cM_i=(M_i,{\Sgq}_{\vert {M_i}})$ are defined over $\what \bK$
and are pure of slope $\mu_i$ and rank $r_i$.
\par
Each $\cM_i$ is a direct sum of
almost admissible indecomposable
$q$-difference modules, \ie iterated non trivial extension
of a simple almost admissible $q$-difference module by itself.
\end{thm}

\begin{rmk}
Irreducible objects are $q$-difference modules over $\what \bK$ obtained by rank one modules
associated to $q$-difference equations of the form $x^\mu y(qx)=\la y(x)$, with
$\mu\in\Q$ and $\la\in\bC^\ast$, by restriction of scalars.
\end{rmk}

Hence the first part of Theorem \ref{thm:mainthm} can be proved:

\begin{prop}\label{prop:formalVSanalytic}
Let $\cM=(M,\Sgq^M)$ and $\cN=(N,\Sgq^N)$ be two almost admissible
$q$-difference modules over $\bK$. Then $\cM$ is isomorphic to $\cN$
over $\bK$ if and only if $\cM_{\what \bK}$ is isomorphic to $\cN_{\what
\bK}$ over $\what \bK$.
\end{prop}

\begin{proof}
It follows from the analytic (resp. formal) factorizations of
$q$-difference modules over $\bK$ (resp. $\what \bK$) that:
$$
\cM\cong\cN\Leftrightarrow\cM_{\bK_n}\cong\cN_{\bK_n}
\hskip 10 pt\hbox{and}\hskip 10 pt
\cM_{\what \bK}\cong\cN_{\what \bK}\Leftrightarrow\cM_{\what \bK_n}\cong\cN_{\what \bK_n}\,,
$$
for an integer $n\geq 1$ such that the the slopes of the two modules become integral
over $\bK_n$.
So we can suppose that the two modules are actually admissible.
\par
If $\cM$ and $\cN$ are isomorphic over $\bK$ than they are necessarily
isomorphic over $\what \bK$. On the other side suppose that
$\cM_{\what \bK}\cong\cN_{\what \bK}$. Then the results follows
from the fact that any formal factorization must actually be analytic
(\cf Corollary \ref{cor:adams}).
\end{proof}

For further reference we point out that we have proved the following statement:

\begin{cor}\label{cor:vectorspace}
Let $\cM=(M,\Sgq)$ be a pure $q$-difference module over $\what\bK$ (resp. a pure
almost admissible $q$-difference module over $\bK$), of slope $\mu$ and rank $\nu$.
Then for any $n\in\Z_{\geq 1}$ such that $n\mu\in\Z$, there exists a
$\bC$-vector space $V$ contained in $M_{\what\bK_n}$
(resp. $M_{\bK_n}$), of dimension $\nu$, such that
$x^{\mu}\Sgq(V)\subset V$.
\end{cor}

\subsection{End of the proof of Theorem \ref{thm:mainthm}}

Theorem \ref{thm:mainthm} states that $\cB_q^{iso}=$\qdiffKaa. Proposition
\ref{prop:formalVSanalytic} implies that \qdiffKaa\ is a subcategory of $\cB_q^{iso}$.
To conclude it is enough to prove the following lemma:

\begin{lemma}
Let $\cM\in\cB_q$. We suppose that any $\cN\in\cB_q$ such that
$\cM_{\what\bK}\cong\cN_{\what\bK}$  in $\what\cB_q$
is already isomorphic to $\cM$ in $\cB_q$.
Then $\cM$ is almost admissible.
\end{lemma}

\begin{proof}
With no loss of generality,
can suppose that the slope of the Newton polygon of $\cM$ are integral.
We know that the lemma is true for rank one modules.
In the general case we prove the lemma by steps:
\begin{description}

\item
[{\it Step 1. Pure rank 2 modules of slope zero.}]
Let us suppose that $\cM$ is pure with Newton polygon $\{(0,2)\}$.
Then there exists a basis $\ul e$ of $\cM_{\what\bK}$ such that
$\Sgq\ul e=\ul e A$, with $A\in Gl_2(\bC)$ in the Jordan normal form.
The assumptions of the lemma
actually say that the basis $\ul e$ can chosen to be a basis of $\cM$ over $\bK$.
If $\cM$ has only one exponent modulo $q^\Z$, then $\cM$ is admissible.
So let us suppose that $\cM$ has at least two different exponents modulo
$q^\Z$: $\a,\be\in\bC$. An elementary manipulation on the exponents
(\cf Remark \ref{rmk:ManipSlopes}) allows to assume that $\be=1$.
This means that $A$ is a diagonal matrix of eigenvalues $1,\a$.
We are in the case of \S\ref{subsec:example}, so we already know
that there exists only one isoformal analytic isomorphism class
if and only if the module is admissible.

\item
[{\it Step 2. Proof of the lemma in the case of a pure module of slope zero.}]
Let us suppose that $\cM$ is pure with Newton polygon $\{(0,r)\}$.
Then there exists a basis $\ul e$ of $\cM$ over $\bK$ such that
$\Sgq\ul e=\ul e A$, with $A\in Gl_r(\bC)$ in the Jordan normal form.
If $\cM$ has only one exponent modulo $q^\Z$, then $\cM$ is admissible.
So let us suppose that $\cM$ has at least two different exponents modulo
$q^\Z$. For any couple of exponents $\a,\be$, distinct modulo $q^\Z$, the
module $\cM$ has a rank two submodule isomorphic to the module consider in step 1.
This implies that $\phi_{q,\a\be^{-1}}$ is convergent and hence that $\cM$ is admissible.

\item
[{\it Step 3. General case.}]
Let $\{(r_i,\mu_i):i=1,\dots,k\}$ be the Newton polygon of $\cM$.
The formal module $\cM_{\what\bK}$ admits a basis $\ul e$ such that
the matrix of $\Sgq$ with respect to $\ul e$ is a block diagonal matrix
of the form (\cf Corollary \ref{cor:vectorspace}):
$$
\Sgq\ul e=\ul e
\hbox{~diag}\begin{pmatrix}
\ds\frac{A_1}{x^{\mu_1}}&\cdots&\ds\frac{A_k}{x^{\mu_k}}
\end{pmatrix}\,,
$$
where $A_1,\dots,A_k$ are constant square matrices that we can suppose to be in Jordan normal form.
The assumption actually says that $\cM$ is isomorphic in $\cB_q$
to the $q$-difference module $\cN$ over $\bK$ generated by
the basis $\ul e$.
Since the slopes and the classes modulo $q^\Z$ of the exponents are
both analytic and formal
invariants, it is enough to prove the statement for pure modules.
If $\cM$ is pure, the statement is deduced by step 2, by elementary manipulation of the
slopes (\cf Remark \ref{rmk:ManipSlopes}).
\end{description}
This ends the proof of the lemma and therefore the proof of Theorem \ref{thm:mainthm}.
\end{proof}

\section{Structure of the category $\cB_q^{iso}$.
Comparison with the results in \cite{BaranovskyGinzburg} and
\cite{SoibelmanVologodsky}}
\label{sec:Soibelman}
%
%
%

The formal results above give another proof of the following:

\begin{thm}[{\cite[Thm. 3.12 and Thm. 3.14]{SoibelmanVologodsky}}]
The subcategory $\what\cB_q^f$ of $\what\cB_q$ of pure $q$-difference modules
of slope zero
is equivalent to the category of $\bC^\ast/q^\Z$-graded finite dimensional
$\bC$-vector spaces equipped with a nilpotent operator that preserves the grading.
\par
The category $\what\cB_q$ is equivalent to the category of $\Q$-graded objects of $\what\cB_q^f$.
\end{thm}

Let $\cB_q^{iso,f}$ be the full subcategory of $\cB_q^{iso}$ of pure
$q$-difference modules of slope zero.
We have an analytic version of the result above:

\begin{thm}\label{thm:Qgraded}
The category $\cB_q^{iso}$ is equivalent to the category of $\Q$-graded objects
of $\cB_q^{iso,f}$ \ie each object of $\cB_q^{iso}$ is a direct sum indexed on $\Q$
of objects of $\cB_q^{iso,f}$ and the morphisms of $q$-difference modules respect the grading.
\end{thm}

\begin{proof}
For any $\mu\in\Q$, the component of degree $\mu$ of an object of
$\cB_q^{iso}$ is its maximal pure submodule of slope $\mu$. The theorem follows
from the remark that there are no non trivial morphisms between two pure modules of different slope.
\end{proof}

As far the structure of the category $\cB_q^{iso,f}$ is concerned
we have an analytic analog of \cite[Thm.3.14]{SoibelmanVologodsky} and
\cite[Thm. 1.6${}^\p$]{BaranovskyGinzburg}:

\begin{thm}\label{thm:C/qZgraded}
The category $\cB_q^{iso,f}$ is equivalent to the category of finite dimensional
$\C^\ast/q^\Z$-graded complex vector spaces $V$ endowed with nilpotent operators
which preserves the grading, that moreover have the following property:
\begin{quote}
$(\mathcal D)$
Let $\la_1,\dots,\la_n\in\C^\ast$ be a set of representatives of the classes
of $\C^\ast/q^\Z$ corresponding to non zero homogeneous components of $V$.
The series $\Phi_{(q;\ul\La)}(x)$, where
$\ul\La=\{\la_{i}\la_{j}^{-1}~:~ i,j=1,\dots,r\,;\,\,\la_{i}\la_{j}^{-1}\not\in q^{\Z_{\leq 0}}\}$,
is convergent.
\end{quote}
\end{thm}

\begin{proof}
We have seen that a module $\cM=(M,\Sgq)$
in $\cB_q^{iso,f}$ contains a $\C$-vector space $V$, invariant under $\Sgq$, such that
$M\cong V\otimes\bK$. Hence there exists a basis $\ul e$, such that $\Sgq\ul e=\ul eB$, with
$B\in Gl_\nu(\C)$ in the Jordan normal form. This means that $B=D+N$, where $D$ is a diagonal constant matrix and
$N$ a nilpotent one. The operator $\Sgq-D$ is nilpotent on $V$.
\par
Since any eigenvalue $\la$ of $D$ is uniquely determined modulo $q^\Z$,
we obtain the $\C^\ast/q^\Z$-grading, by considering the kernel of the operators
$\l(\Sgq-\la\r)^n$, for $n\in\Z$ large enough.
\end{proof}

\newcommand{\noopsort}[1]{}


\end{document}